\documentclass[11pt,a4paper]{amsart}

\usepackage{geometry}
\pdfoutput=1
\usepackage{amsmath,amsfonts,amssymb,amsthm}
\usepackage{xcolor}
\usepackage{todonotes}

\usepackage{pdfpages}
\usepackage[normalem]{ulem}
\usepackage{url}
\usepackage{enumerate}
\usepackage{verbatim}
\usepackage{mathrsfs}
\usepackage{csquotes}
\usepackage{mathtools}
\usepackage{lipsum}
\usepackage{enumitem}

\usepackage{amsmath}        
\usepackage{amsfonts}       
\usepackage{amsthm}         
\usepackage{bbding}         
\usepackage{bm}             
\usepackage{graphicx}       
\usepackage{fancyvrb}       

\usepackage[compress,square,numbers]{natbib}         

\usepackage{hyperref}
\hypersetup{
	colorlinks   = true,
	pdfauthor    = {Vitezslav Kala and Lucien Sima},
	pdftitle     = {},
	pdfkeywords  = {},
	pdfcreator   = {Pdflatex},
	pdfpagemode  = UseNone,
	pdfstartview = FitH
}

\usepackage{dcolumn}        
\usepackage{booktabs}       
\usepackage{xcolor}         
\usepackage{comment}
\usepackage{float}

\makeatletter

\theoremstyle{plain}
\newtheorem{thm}{Theorem}[section]

\newtheorem{claim}[thm]{Proposition}
\newtheorem{obs}[thm]{Lemma}
\newtheorem{coro}[thm]{Corollary}

\theoremstyle{definition}
\newtheorem{defn}[thm]{Definition}
\newtheorem{notation}[thm]{Notation}
\newtheorem{question}[thm]{Question}

\newcommand{\Zn}{\Z^{n}}

\newcommand{\mb}{\mathbf}

\newcommand{\N}{\mathbb{N}}
\newcommand{\Z}{\mathbb{Z}}
\newcommand{\Q}{\mathbb{Q}}

\usepackage{tikz}

\newcommand{\isn}{\mathrm{Isol}_n}

\newcommand{\is}{\mathrm{Isol}}

\title{On generators of commutative semifields}

\author{V\' \i t\v ezslav Kala}
\author{Lucien \v S\' \i ma}
\address{
	Charles University,
	Faculty of Mathematics and Physics,
	Department of Algebra,
	Sokolov\-sk\'a 83, 18600 Praha 8,
	Czech Republic
}
\email{
	vitezslav.kala@matfyz.cuni.cz\\ 
	luciensima@gmail.com
}

\begin{document}

	\begin{abstract}
		We study ideal-simple commutative semirings and summarize the results giving their classification, in particular when they are finitely generated. In the principal case of (para)semifields, we then consider their minimal number of generators and show that it grows linearly with the depth of an associated rooted forest.
	\end{abstract}

	\maketitle

\section{Introduction}

Semirings and semifields are a natural generalization of rings and fields, which found its applications in various areas of mathematics including cryptography, theoretical computer science, and tropical geometry \cite{droste,G,ims,ir,litvinov,mmr,monico,M,zumbragel}. 
A number of recent works \cite{IKN,JK,JKK,Kala15,KK,KKK,KaKo1,KalKor,KNZ,KL, Le} focused on the study of simple semirings and semifields, in particular on the finitely generated ones that provide rich and interesting structure.

To be more precise, recall that a \emph{(commutative) semiring} is a set $S$
equipped with two binary associative and commutative operations, addition $\oplus$ and multiplication $\cdot$, such that multiplication distributes over addition (in this paper, all semirings will be commutative, i.e., multiplication is always assumed to be commutative). 
Semirings offer a natural extension and generalization of (commutative) rings, and, in particular, many of the structural results on rings carry over. Of particular interest for us will be generalizations of simple rings, i.e., those possessing no non-trivial ideals. Of course, every simple (commutative) ring is just a field, and ideals in rings correspond to congruences.

This correspondence no longer holds in semirings, and so one distinguishes \textit{congruence-simple} and \textit{ideal-simple}
semirings. 
While congruence-simple semirings are quite well understood thanks to the seminal paper \cite{BHJK} (with the exception of subsemirings of the positive real numbers), ideal-simple ones remain more mysterious, and so we focus on them.
Their structure can be quite quickly be reduced to that of (para)semifields (see Theorems \ref{classidsimple} and \ref{classem2}):

A semiring $S$ is a \textit{semifield} if moreover there is an element $0$ such that the set of non-0 elements $(S\setminus\{0\},\cdot)$ forms a group and $0\cdot s=0$ for each $s\in S$, and a \textit{parasemifield} if $(S,\cdot)$ is a group.

Specifically, we will be interested in \textit{finitely generated} ideal-simple semirings, motivated by the folklore result  that \textit{if a field is finitely generated as a ring, then it is finite}.

Semiring generalizations of this result have been recently quite intensely studied \cite{JKK, KK, KalKor};
the first of the goals of this article is to summarize the various classification results on finitely generated ideal-simple semirings and (para)semifields, as they have been spread throughout the literature. 

Our main goal is then to focus on the most interesting case of additively idempotent parasemifields. We use their combinatorial classification \cite{Kala15} in terms of rooted forests to obtain new 
results on their minimal numbers of semiring generators. As the classification states that we need to study the abelian groups $(\Z^n,+)$ equipped with suitable partial orders coming from the graph structure of a rooted forest $F$ on $n$ vertices, our proofs entail mostly elementary, but tricky and non-trivial arguments.
Surprisingly, it turns out that the minimal number of generators grows linearly with the depth of $F$ (see Theorem \ref{bounds}). However, determining the precise value seems to be very hard and it is unclear if the answer will depend on the specific structure of $F$, or only on its depth.

These results are interesting not only on their own, but also because additively idempotent parasemifields are term-equivalent with lattice-ordered groups ($\ell$-groups for short, see Section 2 for details). The study of $\ell$-groups is another rich area of great interest \cite{AF, GH, Yang}, and exploiting this connection was already crucial in the previous classification results \cite{Kala15,KalKor} that relied on the classification of Busaniche, Cabrer, and Mundici \cite{BCM}. In fact, in most of the present paper we also use the $\ell$-group notation.
Very notable is also the correspondence with MV-algebras and related topics in logic \cite{BDN-F, BDNF-B, BDN-E,DNG,  GRS, M,  NL}.

As for the contents of this short paper, in Section 2, we summarize the classifications of ideal-simple semirings, of semifields, and of finitely-generated semifields, following \cite{BHJK} and \cite{JK}. 
Their structures in turn depend on parasemifields, and so accordingly in Section 3, we state the classification of parasemifields that are finitely generated as semirings (Theorem \ref{classtree}) 
from \cite{Kala15}: 
Each such parasemifield can be associated with a rooted forest carrying an additive group of integers on each vertex. The second semiring operation $\vee$ is defined as a lexicographic maximum with respect to the forest structure.

Sections 4 and 5 concern the minimal number of generators needed to generate a given parasemifield using the semiring operations. We show that this number is linear in the depth of the rooted forest that represents it (Theorem \ref{bounds}). Among other results, we also give the precise minimal number of generators for the parasemifields corresponding to $\Zn$ equipped with coordinate-wise addition and maximum (Theorem \ref{genzn}; \textit{an elementary reformulation of this surprising result was selected for the shortlist of problems for the International Mathematical Olympiad 2022}). 
We conclude the article with Open Question \ref{que} that offers a possible precise value for the number of generators.

\section*{Acknowledgments}
We thank Miroslav Korbel\' a\v r and Ji\v r\'\i\ \v S\'\i ma for helpful discussions and suggestions.

\section{Preliminaries and basic classifications}

Already in the Introduction we have recalled that a \emph{semiring} $(S,\oplus,\cdot)$ consists of a set $S$
equipped with two binary associative and commutative operations, addition $\oplus$ and multiplication $\cdot$, such that multiplication distributes over addition. Moreover, a semiring $S$ is a \textit{semifield} if there is an element $0$ such that $(S\setminus\{0\},\cdot)$ forms a group and $0 s=0$ for all $s\in S$, and a \textit{parasemifield} if $(S,\cdot)$ is a group; in both cases we denote the unit element $1$ and the inverse $^{-1}$.

A semiring $S$ is \textit{finitely generated} if there are elements $s_1,\dots,s_n$ for some positive integer $n$ such that the smallest subsemiring of $S$ containing  $s_1,\dots,s_n$ equals $S$ itself.
If a semifield is finitely generated as a semiring (i.e., using only the operations $\oplus, \cdot$, but not the inverse $^{-1}$), then we will call it an \textit{fg-semifield}, and similarly in the case of an \textit{fg-parasemifield}.

A semiring $S$ is \textit{additively idempotent} if $s\oplus s=s$ for all $s\in S$. 
Such semirings are studied in tropical mathematics 
where the semiring with its operations is commonly denoted as $(S,\vee, +)$, i.e., $\vee$~denotes the addition and $+$ the multiplication (we will also frequently use this notation).

Of particular interest is the case of additively idempotent parasemifields $(S,\vee,+)$, for we have a term-equivalence with {lattice-ordered groups} ({$\ell$-groups} for short) $(L,+,\vee,\wedge)$. 
Recall that an (abelian) \textit{$\ell$-group} $(L,+,\vee,\wedge)$ is an abelian group
$(L,+)$ that is also a lattice $(L,\vee,\wedge)$ such that $+$ distributes over the lattice operations $\vee,\wedge$. The term-equivalence between additively idempotent parasemifields $(S,\vee,+)$ and $\ell$-groups $(S,+,\vee,\wedge)$ is given by $a\wedge b=-((-a)\vee(-b))$ (and the operations $\vee,+$ staying the same).
For more details see, e.g.,  \cite{AF, GH, W, WW}.

\medskip

\textbf{Convention.} Let us stress that throughout the paper, all semirings, semifields,  parasemifields, and $\ell$-groups are commutative.

\medskip

An \textit{ideal} $I$ in a semiring $S$ is a non-empty subset such that $a\oplus b, s\cdot a\in I$ for all $a,b\in I, s\in S$.
A semiring $S$ is \textit{ideal-simple} if all ideals $I$ in $S$ satisfy $ |I|\leq 1$ or $I=S$.

Recall that for a positive integer $n$, we denote $\Z^n$ the direct product of $n$ copies of $\Z$. We will often consider it as a group or semigroup $(\Z^n,+)$ when equipped with coordinate-wise addition $+$. 
For a vector $\mb v=(v_1,\dots,v_n)\in\Z^n$ and an integer $k$, we will denote 
$k\cdot\mb v$ (or just $k\mb v$) the vector $(kv_1,\dots,kv_n)$.

Now we can summarize the classification results for ideal-simple semirings and their relation to the property of being finitely generated, originally established in \cite{BHJK, JK, KK}. While we do not give the proofs here, they are available in the original articles, or in the thesis \cite{Sima}.

\begin{thm}[{\cite[Theorem 11.2]{BHJK}}] \label{classidsimple}
	Let $S$ be a semiring, $ |S| \geq 3$. Then $S$ is ideal-simple if and only if one of the following cases holds:
	\begin{enumerate}[label=$(\arabic*)$]
		\item $(S,\oplus )$ is isomorphic to the $p$-element cyclic group $(\Z_p,+)$ equipped with zero-multiplication for a prime  $p>3$,
		\item $S$ is a semifield,
		\item $S$ is a parasemifield.
	\end{enumerate}
\end{thm}

Let us  further state the complete classification of semifields.

\begin{thm}[{{\cite[Section 12]{BHJK}}, \cite[Theorem 8.15]{JK}}] \label{classem2}
	Let $S$ be a semifield. Then one of the following cases occurs:
	\begin{enumerate}[label=$(\arabic*)$]
		\item $S$ is a field.
		\item $S$ is constructed from a parasemifield $(T,\oplus ,\cdot)$ by adding an element $0$ and letting $0\oplus s=s$ and $0s=0$ for every $s \in S$.
		\item $S$ is constructed from a multiplicative abelian group $(A,\cdot)$ by adding an element $0$ and letting $s\oplus t=0$ and $0s = 0$ for every $s,t \in S$.
		\item $S$ is constructed from a parasemifield $(P,\oplus ,\cdot)$ as follows: Assume that $(P,\cdot)$ is a multiplicative subgroup of an abelian group $(A,\cdot)$ and let $S =A \cup \{0\}$ and $0s=0$ for every $s\in S$. 
		
		The addition is defined for any $x,y \in S$ as follows:
		\begin{align*}
			x\oplus 0 &= 0 \\
			\text{if\ \  }x^{-1}y \notin P, \text{\ then\ \ } x\oplus y &= 0  \\
			\text{if\ \  }x^{-1}y \in P, \text{\ then\ \ }  x\oplus y &= (x^{-1}y \oplus  1)\cdot x.
		\end{align*}
	\end{enumerate}
\end{thm}

Note that in case (4) above (as well as in Theorem \ref{classemfg}(4) below) we allow $P$ to be the trivial one-element parasemifield -- this gives precisely the semifields denoted as $V(A)$ in \cite[Theorem 8.15]{JK}.

The classification above tells us that semifields arise from well-known structures (fields and groups) or from parasemifields. 

By Theorem \ref{classidsimple}, we see that every finitely generated, ideal-simple semiring (except for the trivial case (1)) is an fg-semifield or fg-parasemifield.

Further, it is not hard to check that the structure of fg-semifields nicely corresponds to the classification from Theorem \ref{classem2}.

\begin{thm}[{\cite[Section 4]{KK}, \cite[Section 4]{JK}}] \label{classemfg}
	Let $S$ be an \emph{fg}-semifield. Then one of the following cases occurs:
	\begin{enumerate}[label=$(\arabic*)$]
		\item $S$ is a \emph{finite} field.
		\item $S$ is constructed from an \emph{fg}-parasemifield $P$ by adding an element $0$ and letting $0\oplus s=s$ and $0s=0$ for every $s \in S$.
		\item $S$ is constructed from a \emph{finitely generated} multiplicative abelian group $(A,\cdot)$ by adding an element $0$ and letting $s\oplus t=0$ and $0s = 0$ for every $s,t \in S$.
		\item $S$ is constructed from an \emph{fg}-parasemifield $(P,\oplus ,\cdot)$ as follows. Let $(P,\cdot)$ be a subgroup of a \emph{finitely generated} abelian group $(A,\cdot)$ and let $S = A \cup \{0\}$ and $0s=0$ for every $s\in S$. 
		
		The addition is defined for any $x,y \in S$ as follows:
		\begin{align*}
			x\oplus 0 &= 0 \\
			\text{if\ \  }x^{-1}y \notin P, \text{\ then\ \ } x\oplus y &= 0  \\
			\text{if\ \  }x^{-1}y \in P, \text{\ then\ \ }  x\oplus y &= (x^{-1}y \oplus  1)\cdot x.
		\end{align*}
	\end{enumerate}	
\end{thm}

The preceding theorems essentially reduce the classification of finitely generated ideal-simple semiring to that of fg-parasemifields.
While not much appears to be known about parasemifields in general, the case of fg-parasemifields is much better understood, as we discuss in the next section.

\section{Finitely generated parasemifields}

To proceed to the main topic of this paper, let us 
present the classification of fg-parasemifields. 
First, by a non-trivial theorem of Kala and Korbel\' a\v r, we can restrict ourselves to the case of additively idempotent parasemifields.

\begin{thm}[{\cite[Theorem 4.5]{KalKor}}]\label{th:ai}
	Let $S$ be an fg-para\-semi\-field. Then $S$ is additively idempotent. 
\end{thm}

Therefore, it suffices to study only additively idempotent fg-semifields, which is the case considered by Kala \cite{Kala15}. Combining his results with the preceding Theorem \ref{th:ai}, we will be able to state Theorem \ref{classtree} and Corollary \ref{gengroup} for general fg-parasemifields without the idempotency assumption.

However, first we need to associate a parasemifield $G(T,v)$ to a rooted tree $(T,v)$ and to extend this notion to rooted forests. 

Recall that a \textit{rooted tree} $(T,v)$ is a (finite) un-oriented graph $T$ without cycles together with a highlighted vertex $v$, called the \textit{root}. A \textit{rooted forest} $(F,R)$ consists of a graph $F = T_1 \sqcup \dots \sqcup T_k$ that is the disjoint union of finitely many rooted trees $(T_i,v_i)$ with the set of roots $R = \{v_1, \dots, v_k\}$. For a graph $G$, we denote $V(G)$ and $E(G)$ its sets of vertices and edges.
A vertex $v$ in a rooted forest is a \textit{leaf} if $v\not\in R$ and its degree is 1 (i.e., there is precisely 1 edge containing $v$), or if $v\in R$ and its degree is 0.

Two rooted forests are \textit{isomorphic} if there is a bijection between their sets of vertices that preserves the graph structure and permutes the sets of roots.

The \textit{depth} of a vertex $w\in V(F)$ in a rooted forest $(F,R)$ is the largest $k$ such that there exists a path $v = v_1, v_2, \dots, v_k = w$ (i.e., a sequence of distinct vertices such that $(v_i,v_{i+1})$ is an edge for each $i$) from some root $v\in R$ to the vertex $w$. The \textit{depth} of a rooted forest is the maximum of the depths of its vertices.

For a positive integer $n$, we denote $[n]=\{1,2,\dots,n\}$.

\begin{defn}
	Let $(T,v)$ be a rooted tree on $n=|V(T)|$ vertices. We attach a copy of the set of integers $\Z_w$ to each vertex $w \in V(T)$ and define the set $$G(T,v)=\prod_{w\in V(T)}\Z_w=\Z^n$$ (of course, the second equality above depends on fixing a bijection of $V(T)$ with $[n]$, as we will usually do). 
	We refer to the elements of $G(T,v)$ as integer-valued vectors from $\Zn$, each coordinate corresponding to a vertex in $(T,v)$.
	
	Let us now define semiring operations $\vee,+$ on $G(T,v)$. The multiplicative group $(G(T,v), +)$ is given by the coordinate-wise addition in the group $(\Zn ,+)$ (which does not depend on structure of the tree $(T,v)$).
	
	To define the semiring addition $\vee$, let $\mathbf{g}=(g_w), \mathbf{h}=(h_w)$ be two elements from $G(T,v)$. 
	We define $\mathbf{g} \vee \mathbf{h}=(({g} \vee {h})_w)$ as follows:
	For a vertex $w \in V(T)$, let $v = v_1, v_2, \dots, v_k = w$ be the  unique path from the root $v$ to the vertex $w$. If ${g}_{v_i} = {h}_{v_i}$ for all $i \in [k]$, we set $({g} \vee {h})_w = {g}_w = {h}_w$. Otherwise, let $i$ be the smallest index such that ${g}_{v_i} \neq {h}_{v_i}$ and define
	$$({g} \vee {h})_w =
	\begin{cases}
		{g}_{w}  & \text{ if } {g}_{v_i} > {h}_{v_i}\\
		{h}_{w} & \text{ if } {g}_{v_i} < {h}_{v_i}.
	\end{cases}$$
\end{defn}

We can naturally extend the definition to rooted forests.

\begin{defn}
	Let $(F,R)$ be a rooted forest, $F = T_1 \sqcup \dots \sqcup T_k$, with the set of roots $R = \{v_1, \dots, v_k\}$. We define the associated parasemifield $(G(F,R),\vee,+)$ as the direct product of the parasemifields $(G(T_i, v_i),\vee,+)$.
\end{defn}

It turns out that every fg-parasemifield arises from a rooted forest in this way, as the following theorem shows.

\begin{thm}[{\cite[Theorem 4.1]{Kala15}, {\cite[Theorem 4.5]{KalKor}}}] \label{classtree}
	Let $(S,\oplus ,\cdot)$ be an fg-parasemifield. Then there is a rooted forest $(F,R)$ (unique up to isomorphism) such that $(S,\oplus ,\cdot) \simeq (G(F,R),\vee,+)$. 
\end{thm}

This classification result has an immediate corollary 

\begin{coro}[{\cite[Corollary 4.6]{KalKor}}]\label{gengroup}
	Let $S$ be an fg-parasemifield. Then $S$ is finitely generated as a multiplicative semigroup.
\end{coro}
\begin{proof}
	Theorem \ref{classtree} gives us that $S \simeq G(F,R) $ for some rooted forest $(F,R)$. The multiplicative group $(S,\cdot)$ is thus isomorphic to $(\Zn ,+)$ for $n = |V(F)|$. The corollary follows from the fact that $\Zn$ is clearly finitely generated as an additive semigroup (for more details, see Proposition \ref{genplus}).
\end{proof}

As a result, we obtain the following corollary, which might be quite surprising. 

\begin{coro} 
	Let $S$ be an ideal-simple semiring that is finitely generated. Then $S$ is finitely generated as a multiplicative semigroup.
\end{coro}

\begin{proof}
	First, from the classification of ideal-simple semirings (Theorem \ref{classidsimple}), we have that $S$ is either isomorphic to $\Z_p$ with zero-multiplication (which is finite and thus finitely generated as a multiplicative semigroup), or an fg-parasemifield or an fg-semifield. 
	
	Corollary \ref{gengroup} states that fg-parasemifields are finitely generated as multiplicative semigroups, which also implies that fg-semifields of type (2) from Theorem \ref{classemfg} are finitely generated as well. The statement also holds for fg-semifields of the three remaining types, as they are either finite (type (1)) or obtained by adding one element to a finitely generated abelian group (types (3) and (4)).
\end{proof}

\section{Generators for isolated vertices}
In the previous section we discussed that every fg-parasemifield $S$ corresponds to a rooted forest $(F,R)$ (in the sense that $S \simeq G(F,R)$). 
From now on we will thus work with the parasemifields $(G(F,R),\vee,+)$ and accordingly denote the semiring multiplication as $+$, inverse as $-$, neutral element as $0$, and addition as $\vee$, i.e., we will use the $\ell$-group notation.

We turn our interest to determining the minimal number of generators needed to generate $S$ as a semiring, i.e., using only the multiplication $+$ and addition $\vee$, but not the inverse $-$ (nor the neutral element $0$). 

One of our main results will be Theorem \ref{bounds} saying that the minimal number of semiring generators of an fg-parasemifield $S$ is linear in the depth of the corresponding rooted forest $(F,R)$.

In this section, we start with 
some preliminary observations, and then we determine the number of generators in the case of forests without any edges.

\begin{notation} The minimal number of semiring generators of the parasemifield $G(F,R)$ will be denoted  $m(F,R)$.
\end{notation}

The multiplicative group of the parasemifield $G(F,R)$ is $(\Zn ,+)$ (where + is the usual addition + taken coordinate-wise). Thus, it is quite useful to determine the minimal number of vectors needed to generate $\Zn$ as an additive semigroup.

We will denote $\mb{e}_i\in\Zn$ the vector having 1 at the $i$th coordinate and 0 everywhere else, and $\mb 0=(0,\dots,0)$.
For $\mb{u}\in \Zn$, we usually denote its coordinates as $\mb u=(u_1,\dots,u_n)$. Let us call $\mb{u}$ a \textit{positive vector} if ${u}_1>0,\dots,{u}_n>0$, and similarly a \textit{negative vector} if ${u}_1<0,\dots,{u}_n<0$.

\begin{obs} \label{n+1}
	Let $n$ be a positive integer and $\mb{u}\in \Zn$ a negative vector. Then the set $\{\mb{e}_1, \dots \mb{e}_n, \mb{u}\}$ generates the semigroup $(\Zn,+)$.
\end{obs}
\begin{proof}
	Let us take any vector $\mb{v} \in \Zn$. Since $\mb{u}$ is a negative vector, we can find a positive integer $k$ such that $\mb{w} =(w_1,\dots,w_n) = \mb{v} - k \cdot \mb{u}$ is a positive vector. Then $\mb{v} = \mb{w} + k \cdot \mb{u} = w_1 \cdot \mb{e}_1 + \dots + w_n \cdot \mb{e}_n + k \cdot \mb{u}$ for positive integers $k, w_1, \dots, w_n$, as we wanted to show.
\end{proof}

\begin{claim} \label{genplus}
	The minimal number of semigroup generators of $(\Zn,+)$ is $n+1$.
\end{claim}
\begin{proof}
	Lemma \ref{n+1} gives us a set of $n+1$ generators, and so it suffices to show that any $n$ vectors do not generate $\Zn$. For contradiction, suppose that a set $V = \{\mb{v}_1,\dots,\mb{v}_n \}$ generates $\Zn$. In particular, we can find non-negative integers $a_i,b_i$ such that: 
	$$a_1 \cdot \mb{v}_1 + \dots + a_n \cdot \mb{v}_n = \mb{e}_1,\ \ \  b_1 \cdot \mb{v}_1 + \dots + b_n \cdot \mb{v}_n = -\mb{e}_1.$$
Adding these equations, we get a non-trivial linear combination expressing the zero vector
	\begin{align} \label{notrivialcomb}
		(a_1+b_1) \cdot \mb{v}_1 + \dots + (a_n+b_n) \cdot \mb{v}_n = \mb{0}.
	\end{align}
	
	It is easy to see that $V$ generates the vector space $\Q^{n}$ over $\Q$. Because $V$ consists of $n$ vectors and the dimension of $\Q^{n}$ over $\Q$ is $n$, it follows that $V$ is a basis of $\Q^{n}$. But we have found a non-trivial linear combination \eqref{notrivialcomb} that expresses the zero vector,  showing that $V$ is not linearly independent, which is a contradiction.
\end{proof}

In order to bound $m(F,R)$, we will start with the base case when the depth of $F$ equals $1$. We thus consider rooted forests $(F,R)$ consisting of $n$ isolated vertices (then clearly $R = F$, because each tree component consists of exactly one vertex). We will denote such forests by $ \isn$ and give the exact value of $m( \isn)$ in the rest of this section. 

In other words, we are looking for a minimal set $X$ of vectors from $\Zn$ such that $X$ generates all elements of $\Zn$ using addition and maximum (both applied coordinate-wise). We will start with the easy case $n \leq 2$.

\begin{claim} \label{genz12}
	Let $n \in \{1,2\}$. Then $m( \isn) = 2$.
\end{claim}

\begin{proof}
	It is clear that one generator can not be sufficient, as the sign is preserved under both operations, i.e., every coordinate would stay either positive or negative, and so all of $\Zn$ could not be generated by just one generator.

	We finish the proof by finding the generating set $X$ of size two. If $n=1$, we let $X = \{(1),(-1)\}$ that generate $\Z^1$ just using $+$.

	For $n=2$, we define $X$ to be $\{(1,-2),(-2,1)\}$. It suffices to generate the following three vectors $\{(1,0),(0,1),(-1,-1)\}$ since they generate $\Z^2$ using coordinate-wise addition (see Lemma \ref{n+1}). We obtain the first one as follows:
	\begin{align*}
		(5, -2) &=  (1,-2) \vee (5 \cdot (1,-2) )  \\
		(1,0) &= (5,-2) + 2 \cdot (-2,1),
	\end{align*}
	and the second one is obtained symmetrically. 
Finally, $(-1,-1)=(1,-2)+(-2,1)$.
\end{proof}

Let us now consider the case when $n \geq 3$. Surprisingly, it turns out that $m( \isn) = 3$ regardless of the value of $n$. In order to show that two generators do not suffice, we need to state an auxiliary lemma.

\begin{obs}
	Let $1 \leq j \neq k \leq n$ and let us take two vectors $\mb{u},\mb{v} \in G(\isn)=\Zn$ satisfying $u_j \geq a \cdot u_k$ and $v_j \geq a \cdot v_k$ for some positive real number $a$. Then the same inequalities hold for the vectors $\mb{u} + \mb{v}, \mb{u} \vee \mb{v}\in G(\isn)$.
\end{obs}

\begin{proof}
	The inequality for $\mb{u} + \mb{v}$ is verified by an easy computation:
	\begin{align*}
		(u+v)_j = u_j + v_j \geq a \cdot u_k + a \cdot v_k = a \cdot (u+v)_k.
	\end{align*}

	Let $\mb{m} = \mb{u} \vee \mb{v}$. We have $m_j \geq u_j \geq a \cdot u_k$ and that $m_j \geq v_j \geq a \cdot v_k$. Since $m_k = u_k$ or $m_k = v_k$, the conclusion follows.
\end{proof}

\begin{claim} \label{ineqinv}
	Let $X = \{\mb{g}_1, \dots, \mb{g}_m\}$ be a set of vectors in $G(\isn)=\Zn$, $\mb{g}_i=(g_{i,1},\dots,g_{i,n})$. Let $1 \leq j \neq k \leq n$ and let $a$ be a positive real number such that $g_{i,j} \geq a \cdot g_{i,k}$ for every $i\in[m]$. 
	
	Then $X$ does not generate $G(\isn)$.
\end{claim}
\begin{proof}
	By the previous lemma, the inequality $u_j \geq a \cdot u_k$ holds for all vectors $\mb{u}=(u_1,\dots,u_n)$ generated by $X$ -- but this inequality does not hold for every vector in $G(\isn)=\Zn$.
\end{proof}

We are now ready to prove the following theorem.

\begin{thm} \label{genzn}
	Let $n \in \N$, $n \geq 3$. Then $m( \isn) = 3$.
\end{thm}

\begin{proof}
	We will start with showing that two generators do not suffice. For contradiction, suppose that the set $X = \{\mb{u}, \mb{v}\}$ generates $G(\isn)=\Zn$. 
	
	Suppose that there is a coordinate $i$ such that $u_i, v_i \geq 0$. Both operations preserve the sign, thus we can not generate any vector that has negative $i$th coordinate. Similarly if $u_i, v_i \leq 0$.
	
	Therefore, for every $i \in [n]$ we have $u_i > 0, v_i < 0$, or $u_i < 0, v_i > 0$. Since $n \geq 3$, there are two coordinates $j \neq k$ such that $u_j$ has the same sign as $u_k$ and $v_j$ has the same sign as $v_k$. 
	
	Without loss of generality, assume that $u_j,u_k > 0, v_j,v_k < 0$. Let us denote the positive real number $u_j / u_k$ by $a$. If $v_j / v_k \leq a$, then both inequalities $u_j \geq a \cdot u_k$ and $v_j \geq a \cdot v_k$ are satisfied. On the other hand, if $v_j / v_k \geq a$, then both $u_k \geq (1/a) \cdot u_j$ and $v_k \geq (1/a) \cdot v_j$ are satisfied. In either case, we found an inequality satisfied by both vectors from $X$, and so $X$ does not generate $G(\isn)$ by Proposition~\ref{ineqinv}.

	We finish the proof by finding a set of three generators of $G(\isn)=\Zn$. We let $k = n^2 + 1$ and we define $X=\{\mb{a},\mb{b},\mb{c}\}$ by setting 
	$$a_i= i,\ \ \ b_i = k - i^2,\ \ \  c_i = -1.$$ 
	Note that $k$ is chosen so that $\mb{b}$ is a positive vector.
	
	We will start by generating $n$ positive vectors $\mb{u}_1, \dots, \mb{u}_n$ such that the $i$th coordinate of $\mb{u}_i$ is strictly largest. We define $\mb{u}_i$ to be $2i \cdot \mb{a} + \mb{b}$. Then the $j$th coordinate of $\mb{u}_i$ is $u_{i,j} = 2i \cdot a_j + b_j = 2ij + k - j^2$ = $k + j (2i-j)$. It is easy to see that this expression attains maximum for $j=i$, which gives us that the $i$th coordinate of $\mb{u}_i$ is indeed maximal, i.e., $u_{i,i}>u_{i,j}$ for all $j\neq i$.
	
	For every $i \in [n]$, let us then generate $\mb{v}_i = \mb{u}_i + ((u_{i,i}) - 1) \cdot \mb{c}$. Note that the $i$th coordinate of $\mb{v}_i$ equals $1$ and all the other coordinates are non-positive. 
	
	Now it is the time to apply the coordinate-wise maximum  $\vee$. We obtain the zero vector as $\mb{0} = (\mb{v}_1 \vee \mb{v}_2 \vee \dots \vee \mb{v}_n) + \mb{c}$. Finally, we get $\mb{e}_i$ as $\mb{v}_i \vee \mb{0}$. 
	
	By Lemma $\ref{n+1}$, all the vectors $\mb{e}_i$ together with the negative vector $\mb{c}$ generate $(\Zn,+)$ as a semigroup.
\end{proof}

\section{Minimal number of generators}
Before we study the generators of general forests, let us introduce a partial ordering $\preceq$ on the set of rooted forests, which is compatible with the function $m$ in the sense that $(F,R) \preceq (E,S)$ implies $m(F,R) \leq m(E,S)$.

\begin{defn}
	Let $(F,R)$ and $(E,S)$ be two rooted forests. We say that $(F,R) \preceq (E,S)$ if $(F,R)$ can be obtained from $(E,S)$ by repeatedly deleting leaves (and the edges that connected them to the forest) from the forest $E$. Note that $R \subseteq S$ is the set of roots from $S$ which were not deleted.
\end{defn}

\begin{obs} \label{subtree} 
	Let $\preceq$ be the relation on the set of rooted forests defined as above. Then  $(F,R) \preceq (E,S)$ implies that $m(F,R) \leq m(E,S)$.
\end{obs}
\begin{proof} 
	Let $(E_1,S_1)$ be a rooted forest obtained by deleting a leaf $l$ from $(E,S)$.
	If $X = \{\mb{g}_1, \dots, \mb{g}_k\}$ is a minimal generating set of $G(E,S)$, then we can obtain a generating set of size $k$ for $G(E_1,S_1)$ by simply deleting the coordinate which corresponds to the leaf $l$ from all generators in $X$.
	Thus $m(E_1,S_1) \leq m(E,S)$.
	
	As $(F,R)$ is obtained by repeatedly deleting leaves from $(E,S)$, we can repeatedly use the result of the  previous paragraph to establish $m(F,R) \leq m(E,S)$.
\end{proof}

We will now determine the minimal number of semiring generators for parasemifields which correspond to rooted paths, i.e., to rooted trees that have exactly one leaf. 

\begin{claim} \label{genpath}
	Let $P_n = (\{v_1, \dots, v_n\},v_1)$ be a rooted path with the root $v_1$. Then $m(P_n) = n+1$.
\end{claim}
\begin{proof}
	It follows from the definition of the operation $\vee$ that we have either $\mb{v} \vee \mb{w} = \mb{v}$ or $\mb{v} \vee \mb{w} = \mb{w}$ for every $\mb{v},\mb{w} \in G(P_n)$. Therefore, the minimal number of semiring generators of $G(P_n)$ equals the minimal number of semigroup generators of $\Zn$, which is $n+1$ by Proposition \ref{genplus}.
\end{proof}

\begin{coro} \label{bounddepth}
	Let $(F,R)$ be a rooted forest of depth $l$. Then $m(F,R) \geq l+1$. 
\end{coro}
\begin{proof}
	From the definition of the depth, we can find $w \in V(F)$ and $v \in R$ such that there is a path $P$ from $v$ to $w$ consisting of $l$ vertices. It follows that $(F,R) \succeq (P,v)$. Combining Lemma \ref{subtree} and Proposition \ref{genpath}, we obtain that $m(F,R) \geq m(P,v) = l+1$.
\end{proof}

Let further $(F,R)$ be a general rooted forest. 

We have two parameters for measuring the size of $(F,R)$, namely its depth and its amount of branching, captured by the number of roots and by
the degrees of vertices. Specifically, let us define the \textit{width} of a rooted forest to be the maximum of the number of roots, degrees of the roots, and of the degrees of all vertices $-1$, i.e.,
$$\text{width}(F,R)=\max\left\{|R|,{\max}_{r\in R}\{\text{deg}(r)\},{\max}_{v\in V(F)}\{\text{deg}(v)-1\}\right\},$$
where $\text{deg}(v)$ denotes the degree of a vertex $v$.
As we have just seen in Corollary \ref{bounddepth}, $m(F,R)$ grows at least linearly with the depth of $(F,R)$  but, on the other hand, rooted forests of arbitrarily large width can still have constant $m(F,R)$: by Theorem \ref{genzn}, $m( \isn) = 3$ for any $n \geq 3$ (and one could easily modify this example, e.g., to a tree with a root connected to $n$ leaves).  
Let us thus define a `universal' rooted forest of width $k$ and depth $l$.

\begin{defn}
	Let $k,l$ be positive integers. We define $T_{kl}$ as the unique rooted forest such that:
	\begin{itemize}
		\item there are $k$ roots and each of them has degree $k$,
		\item every vertex that is not a leaf or a root has degree $k+1$,
		\item every leaf has depth exactly $l$. 
	\end{itemize}  
\end{defn}

To illustrate the definition, we give the following picture containing two examples of what $T_{kl}$ looks like.  

\begin{figure}[H]
	\centering
	\includegraphics[width = 9cm]{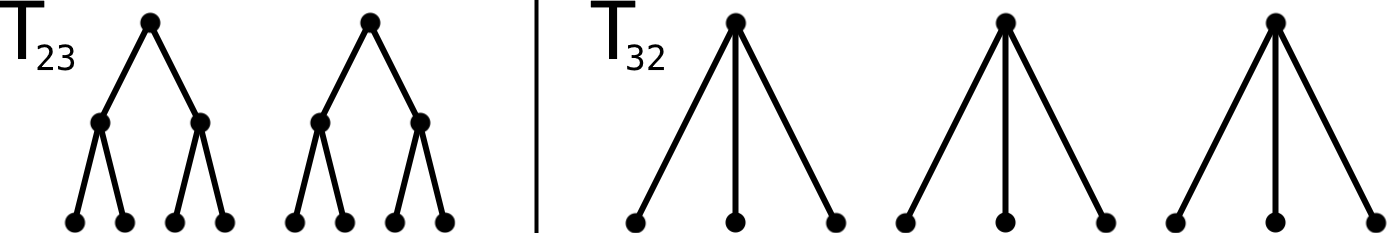}
	\caption{Rooted forests $T_{23}$ and $T_{32}$}
\end{figure}

Observe that $G(F,R) \preceq T_{kl}$ for every rooted forest $(F,R)$ of width $k$ and depth $l$, which implies that $m(F,R) \leq m(T_{kl})$ by Lemma \ref{subtree}. In order to give an upper bound on $m(F,R)$, we would like to estimate $m(T_{kl})$ from above. The first step is the following quite powerful theorem.

\begin{thm} \label{boundcopies}
	Let $(F,R)$ be a rooted forest, $k = m(F,R)$, and construct a rooted forest $(E,S)$ as the disjoint union of $m \geq 2$  copies of $(F,R)$.  
	
	(a) If $m=2$, then $m(E,S) \leq k+1$.
	
	(b) If $m \geq 3$, then $m(E,S) \leq k+2$.
\end{thm}

\begin{proof}
	Let $n = | V(F)|$ and identify the sets $\Z^n=G(F,R)$ and $(\Zn)^m=G(E,S)$. Let  $X = \{\mb{g}_1, \dots, \mb{g}_k \} \in \Zn$ be a minimal set of generators of $G(F,R)$. Throughout the proof, we are going to work with vectors from $(\Zn)^m$ and we will denote $\mb{v} \in (\Zn)^m$ by $(v_1, \dots, v_n \mid v_{n+1}, \dots, v_{2n} \mid \dots \mid v_{(m-1)n+1}, \dots, v_{mn})$.
	For an integer $a\in\Z$, we will denote $\mb{a}=(a,\dots,a)$.

	We shall start with part (a), i.e., $m=2$. Let  $C$ be an integer such that $C > |g_{i,j}|$ for all $i \in [k], j \in [n]$. Define the set $H = \{\mb{h}_1, \dots, \mb{h}_{k+1}\}$ of $k+1$ vectors from $(\Z ^ {n})^2$. The first $k$ of them are defined as $\mb{h}_i = (\mb{g}_i + 2\mb{C} \mid \mb{g}_i-4\mb{C})$ and we let $\mb{h}_{k+1} = (-\mb{C} \mid 2\mb{C})$. 
	
	Our goal will be to show that we are able to generate any vector from $(\Z ^ {n})^2$ from the set $H$. By adding $2 \cdot \mb{h}_{k+1}$ to each $\mb{h}_i$, we obtain $(\mb{g}_i \mid \mb{g}_i)$. Since $X$ generates $G(F,R)$, we are able to obtain all vectors of the form $(\mb{v} \mid \mb{v})$ for any $\mb{v} \in \Zn$, in particular, the vector $(-2\mb{C} \mid -2\mb{C})$. 
	
	Since $C > |g_{i,j}|$ for all $i,j$, we have that $\mb{h}_i \vee \mb{h}_{k+1} = (\mb{g}_i + 2\mb{C} \mid 2\mb{C})$. Adding $(-2\mb{C} \mid -2\mb{C})$ to $(\mb{g}_i + 2\mb{C} \mid 2\mb{C})$, we obtain $(\mb{g}_i \mid \mb{0})$ for each $i \in [k]$, which suffices to generate $(\mb{v} \mid \mb{0})$ for each $\mb{v} \in \Zn$. 
	
	Finally, any vector $(\mb{v} \mid \mb{w}) \in (\Z ^ {n})^2$ can be constructed by adding $(\mb{v-w} \mid \mb{0})$ to $(\mb{w} \mid \mb{w})$, and so we are done with the first part.

	For part (b), we define the set $H = \{\mb{h}_1, \dots, \mb{h}_{k+2}\}$ of $k+2$ vectors from $(\Zn)^m$ as follows:
	\begin{align*}
		\mb{h}_i &= (\mb{g}_i \mid \mb{g}_i \mid \mb{g}_i \mid \dots \mid \mb{g}_i), i \in [k] \\
		\mb{h}_{k+1} &= (\mb{1} \mid  \mb{2} \mid \mb{3} \mid \dots \mid \mb{m}) \\
		\mb{h}_{k+2} &= (\mb{m^2+1 - 1^2} \mid \mb{m^2+1 - 2^2} \mid \mb{m^2+1 - 3^2} \mid \dots \mid \mb{m^2+1 - m^2}) \\
		&= (\mb{m^2} \mid \mb{m^2-3} \mid \mb{m^2-8} \mid  \dots \mid \mb{1}).
	\end{align*}
	Our goal will be to prove that $H$ is a generating set of $(\Zn)^m$. As in the first part, we can use vectors $\mb{h}_1, \dots, \mb{h}_k$ to generate $(\mb{v} \mid \dots \mid \mb{v})$ for any $\mb{v} \in \Zn$, in particular, the vector $\mb{c} = (\mb{-1} \mid \dots \mid \mb{-1})$.
	
	Since vectors $\{(1,2,\dots,m),  (m^2,m^2-3,m^2-8,\dots,1),  (-1,-1,\dots, -1)\}$ generate $\Z^m$ (see the proof of Theorem \ref{genzn}), we are able to use vectors $\mb{h}_{k+1}, \mb{h}_{k+2}, \mb{c}$ to generate $(\mb{c}_1 \mid \mb{c}_2 \mid \dots \mid \mb{c}_m)$ for any integers $c_1, \dots, c_m$. 
	
	Thanks to Lemma \ref{n+1}, it suffices now to generate any vector $\mb{e}_{ni+j}$ for $i \in \{0,\dots,m-1\}, j \in [n]$. Let us take such $i$ and $j$.  Thanks to the previous paragraph, we can generate the vector $\mb{t}_i = (\mb{-1} \mid \dots \mid \mb{-1} \mid \mb{0} \mid \mb{-1} \mid \dots \mid \mb{-1})$ such that $\mb{0}$ lies in the $i$th copy of $G(F,R)$. We then obtain the vector $\mb{u}_{i,j} = \mb{t}_i + (\mb{e}_j \mid \dots \mid \mb{e}_j)$ that has all $n$-tuples non-positive except for the $i$th tuple, which contains $\mb{e}_j$. We then obtain $\mb{e}_{ni+j}$ as ${\mb{u}_{i,j}} \vee \mb{0}$.
\end{proof}

Theorem \ref{boundcopies} plays an important role in establishing the upper bounds for $m(T_{kl})$ in the following theorem.

\begin{thm} \label{boundtkl}
	Let $k,l$ be positive integers. Then
	
(a) $m(T_{1l}) = l+1$,
	
(b) $l+1 \leq m(T_{2l}) \leq 2l$,
	
(c) $l+1 \leq m(T_{kl}) \leq 3l$ for $k \geq 3$.
\end{thm}
\begin{proof}
	For part (a), it is enough to observe that $T_{1l}=P_l$ is actually the path of length $l$ and we already know that $m(T_{1l}) = m(P_l) = l+1$ (Proposition \ref{genpath}).

	Let us prove parts (b) and (c) together. The lower bound follows from Corollary \ref{bounddepth}, as the depth of $T_{kl}$ is  $l$.

	We are going to prove the upper bound by induction on $l$. If $l=1$, then $T_{k1}$ is formed by $k$ isolated vertices, i.e., $T_{k1}=\is_k$. We have already proved (Proposition \ref{genz12} and Theorem \ref{genzn}) that $m(\is_2) = 2$ and $m(\is_k) = 3$ for $k \geq 3$, which gives the upper bound for the case $l=1$.

	We prove the inductive step only for part (c), the other part (b) being very similar. Let us suppose that $m(T_{kl}) \leq 3l$ and we want to show that $m(T_{k(l+1)}) \leq 3l+3$. 
	
	As shown in Figure 2, we can construct $T_{k(l+1)}$ from $T_{kl}$ in two steps. First, we connect all the roots of $T_{kl}$ to a new root $r$, thus creating the rooted tree $U_{kl}$, and then we obtain $T_{k(l+1)}$ as the disjoint union of $k$ copies of $U_{kl}$.

	\begin{figure}[H]
		\centering
		\includegraphics[width = \textwidth]{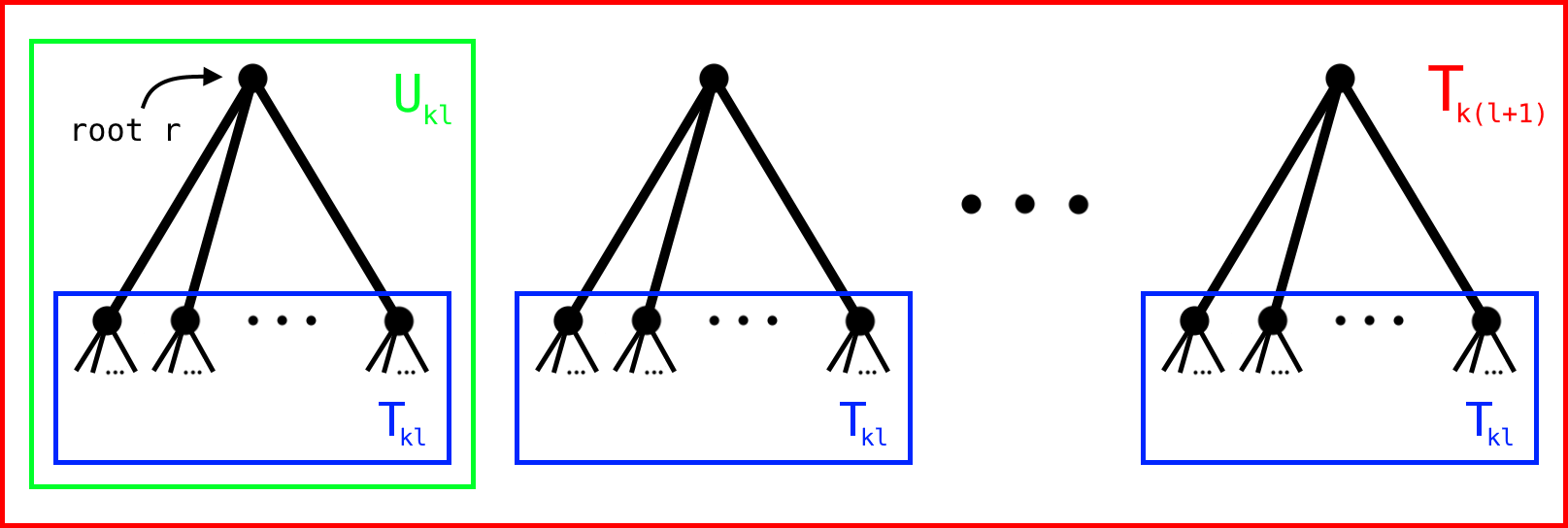}
		\caption{Construction of $T_{k(l+1)}$ from $T_{kl}$.}
	\end{figure}
	
	Let $n = |V(T_{kl})|$. We use the assumption $m(T_{kl}) \leq 3l$ to find a set $X = \{\mb{g}_1, \dots, \mb{g}_{3l}\} \subseteq \Zn$ that generates $G(T_{kl})$ as a semiring. We are going to show that $m(U_{kl}) \leq 3l+1$ by finding its generating set $H = \{ \mb{h}_1,\dots,\mb{h}_{3l+1}\} \subseteq \Z^{n+1}$ consisting of $3l+1$ vectors. We can assume that the first coordinate of these vectors corresponds to the root $r$ of $U_{kl}$.
	
	For any $i \in [3l]$, we let $\mb{h}_i = (-1, \mb{g}_i)$ and we let $\mb{h}_{3l+1} = \mb{e}_1$. Adding the vector $\mb{h}_{3l+1}$ to each $\mb{h}_i$, we obtain $(0,\mb{g}_i)$, which we can be used to generate $(0,\mb{w})$ for any $\mb{w} \in \Zn$ (since $X$ generates $G(T_{kl})$). 
	
	Let $(c,\mb{v}) \in \Z^{n+1}$ be an arbitrary vector, $c \in \Z, \mb{v} \in \Zn$. If $c \geq 0$, we can obtain $(c,\mb{v})$ as $c \cdot \mb{h}_{3l+1} + (0,\mb{v})$. On the other hand, if $c < 0$, then we generate $(c,\mb{v})$ as $(-c) \cdot \mb{h}_1 + (0, \mb{v} - c \cdot \mb{g}_1)$, which finishes the proof that $m(U_{kl}) \leq 3l+1$.

	Since $T_{k(l+1)}$ is constructed as the disjoint union of $k$ copies of $U_{kl}$, the desired bound $m(T_{k(l+1)}) \leq 3l+3$ follows from Theorem \ref{boundcopies} and already proven bound $m(U_{kl}) \leq 3l+1$.
\end{proof}

Now it only takes one last step to give the bounds on $m(F,R)$ for a general rooted forest $(F,R)$. We also obtain a tighter upper bound for binary forests, i.e., in the case of width 2.

\begin{thm} \label{bounds}
	Let $(F,R)$ be a rooted forest of depth $l$. 

(a) We have $l+1 \leq m(F,R) \leq 3l$.

(b) If $(F,R)$ has width $2$, then $l+1 \leq m(F,R) \leq 2l$.
\end{thm}
\begin{proof}
	We prove both parts together. The lower bound follows from Corollary \ref{bounddepth}. Let us denote the width of $(F,R)$ by $k$. It is easy to see that $(F,R) \preceq T_{kl}$, and so $m(F,R) \leq m(T_{kl})$ by  Lemma \ref{subtree}. Using the bounds on $T_{kl}$ from Theorem \ref{boundtkl}, we obtain the result.
\end{proof}

Note that one can prove the upper bound in Theorem \ref{bounds}(a) directly from Theorem \ref{genzn}: Let us sketch the construction of a set $X$ generating $G(F,R)$ such that $ |X| \leq 3l$, where $l$ is the depth of $(F,R)$. 

We split vertices of $(F,R)$ into $l$ disjoint subsets $V_1, \dots, V_l$, where $V_i = \{v \in V(F) \mid \text{depth of }$ $v \text{ is}$ $\text{exactly } i \}$. For each $V_i$, we take (at most) three generators of $G(\is_{n_i})$, where $n_i = |V_i|$  (see Proposition \ref{genz12} and Theorem \ref{genzn}). We set the other coordinates (corresponding to vertices that do not belong to $V_i$) of those generators to 0 and include the resulting vectors in $X$. It can be shown that such $X$ generates $G(F,R)$.

\section{Concluding remarks}

While we have proved that the number of generators $m(F,R)$ grows linearly with the depth of the forest,  
it seems hard to determine the precise value of $m(F,R)$ for all rooted forests $(F,R)$. The following open question suggests a possible answer.

\begin{question} \label{que}
	Let $(F,R) \neq  \isn$ be a rooted forest of depth $l$. Does $m(F,R)$ equal $l+1$?
\end{question}

Note that $m(F,R) \geq l+1$ is true by Corollary \ref{bounddepth}. For Question \ref{que} to have positive answer, it thus suffices to find a generating set of $G(F,R)$ of size $l+1$. We were able to do so for several classes of rooted forests. We end the paper by presenting these partial results.

First, in Figure 3 we give a table of generating sets of parasemifields $G(F,R)$ such that $(F,R)$ contains less than $5$ vertices and $(F,R) \neq  \isn$.
This shows that Question \ref{que} has positive answer for small rooted forests.

\begin{figure}
	\centering
	\includegraphics[width = \textwidth]{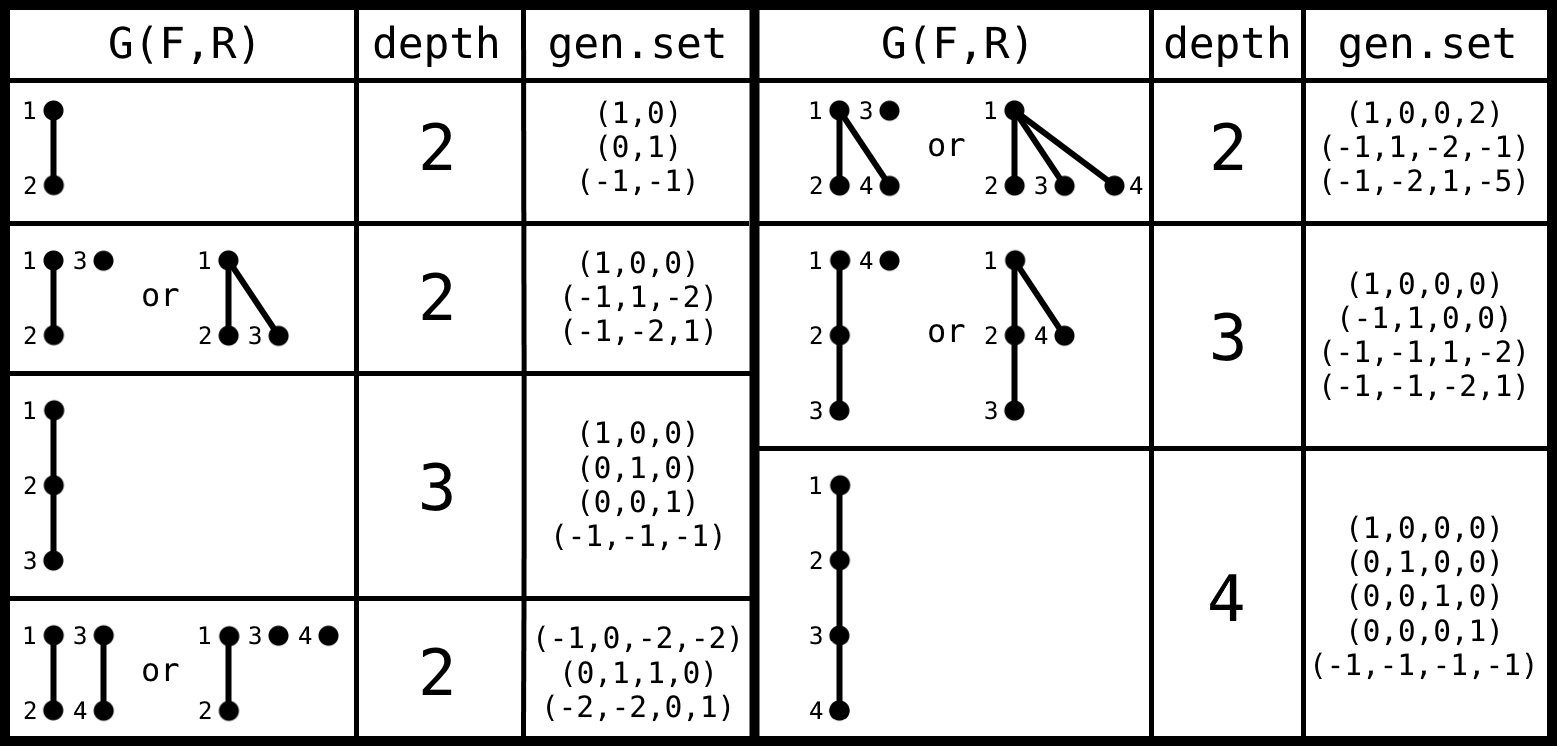}
	\caption{Generators of $G(F,R)$ for $(F,R) \neq  \isn$, $ |V(F)| < 5$. (The numbering of vertices in the rooted forests gives the order of coordinates in the generating sets.)}
\end{figure}
 We are next going to look at rooted forests that are the disjoint union  of several rooted paths and we will answer Question \ref{que} for some of them.

	For  positive integers $k,n$, let us denote by $kP_n$ the rooted forest formed by the disjoint union of $k$ copies of the rooted path $P_n$.

\begin{thm}
	Let $k,n$ be positive integers and $k \leq n+1$. Then there exists a set of $n+1$ generators of $G(kP_n)$. Consequently, $m(kP_n) = n+1$.
\end{thm}

\begin{proof}
	It suffices to prove the theorem for $k = n+1$. Elements from $G((n+1)P_n)$ are vectors from $(\Zn)^{n+1}$ which will be denoted as $\mb{v} = (v_1, \dots, v_n \mid v_{n+1}, \dots, v_{2n} \mid \dots \mid v_{n^2+1}, \dots, v_{n^2+n})$.
	
	We define the set $X = \{\mb{g}_1, \dots, \mb{g}_{n+1}\}$ of $n+1$ vectors from $(\Zn)^{n+1}$ as follows:
	\begin{align*}
		\mb{g}_1 &= (\mb{-2} \mid \mb{e}_1 \mid  \mb{e}_1 \mid \dots \mid  \mb{e}_1 \mid  \mb{e}_1) \\
		 \mb{g}_2 &= ( \mb{e}_2 \mid \mb{-2} \mid  \mb{e}_2 \mid \dots \mid  \mb{e}_2 \mid  \mb{e}_2) \\
		\vdots & \hspace{2.65cm} \ddots \\
		\mb{g}_n &= ( \mb{e}_{n} \mid   \mb{e}_{n} \mid \mb{e}_{n} \mid \dots \mid \mb{-2} \mid   \mb{e}_{n}) \\
		 \mb{g}_{n+1} &= ( \mb{e}_1 \mid  \mb{e}_2 \mid \mb{e}_3 \mid \dots \mid \mb{e}_n \mid \mb{-2}).
	\end{align*}
	
	We are going to show that $X$ generates $G((n+1)P_n)$. Let us start with generating two important vectors.
	\begin{align*}
		\mb{-1} &=  \mb{g}_1 +  \mb{g}_2 + \dots +  \mb{g}_{n+1}\\
		\mb{0} &= ( \mb{g}_1 \vee 2 \cdot  \mb{g}_1) + ( \mb{g}_2 \vee 2 \cdot  \mb{g}_2) + \dots + ( \mb{g}_{n+1} \vee 2 \cdot  \mb{g}_{n+1})
	\end{align*}
	
	We will finish the proof by generating all the vectors $\mb{e}_{ni+j}$ (for $i \in \{0, \dots, n\}$ and $j \in [n]$) and applying Lemma \ref{n+1}. If $i \neq j$, we first generate the vector $\mb{v}_{ij}$ as follows.
	
	\begin{align*}
		\mb{v}_{ij} &= \mb{g}_i + 3 \cdot  \mb{g}_j + \sum_{k \notin \{i,j\}} (2 \cdot  \mb{g}_k) \\
		&=  \mb{g}_j - \mb{g}_i + 2 \sum_{k=1}^{n+1}  \mb{g}_k \\
		&= \mb{-2} +  \mb{g}_j - \mb{g}_i 
	\end{align*}
	
	Observe that the vector $\mb{v}_{ij}$ contains $\mb{e}_j$ in the $i$th $n$-tuple and all the other $n$-tuples contain a non-positive vector. It follows that $\mb{e}_{ni+j} = \mb{v}_{ij} \vee \mb{0}$.

	The approach in the case $i=j$ is similar. We generate:
	\begin{align*}
		\mb{w}_i &= \mb{g}_i + 3 \cdot  \mb{g}_{n+1} + \sum_{k \notin \{i,n+1\}} (2 \cdot  \mb{g}_k) \\
		\mb{e}_{ni+i} &= \mb{w}_i \vee \mb{0}
	\end{align*}
	and we are done.
\end{proof}

It can also be shown that $m(kP_2) = 3$ for any $k \in \N$, using the generating set of $G(\is_{2k})$ from Theorem \ref{genzn}. The proof is similar to the proof of Theorem \ref{genzn}, but slightly more technical.

Finally, the analogous question concerning numbers of generators of \textit{semifields} remains completely open. The answer primarily hinges on the following question that offers rich opportunities for further research.
	
\begin{question} \label{que2}
	Suppose that $S$ is an fg-semifield constructed from a parasemifield $G(F,R)$ and an abelian group $A$ as in Theorem \ref{classemfg}(4). How does the minimal number of semiring generators of $S$ depend on the rooted forest $(F,R)$ and on the embedding in the abelian group $A$?
\end{question}


\begin{thebibliography}{B}


{\footnotesize

\bibitem{AF}
\rm M. Anderson, T. Feil, 
\it Lattice-Ordered Groups, 
\rm  Reidel Texts in the Mathematical Sciences, {\bf 1988}.

\bibitem{BHJK}
\rm R. El Bashir, J. Hurt, A. Jan\v ca\v r\'\i k, T. Kepka,
\it Simple commutative semirings,
\rm J. Algebra {\bf 236} (2001), 277--306.

\bibitem{BDN-F}
\rm L. P. Belluce, A. Di Nola, 
\it Yosida type representation for perfect MV-algebras,
\rm Math. Logic Quart. {\bf 42} (1996), 551--563. 

\bibitem{BDNF-B}
\rm L. P. Belluce, A. Di Nola,  A. R. Ferraioli,
\it MV-semirings and their sheaf representations, 
\rm Order {\bf 30} (2013), 165--179.


\bibitem {BDN-E} 
\rm L. P. Belluce, A. Di Nola, G. Georgescu,
\it Perfect MV-algebras and l-rings,
\rm J. Appl. Non-Classical Logics {\bf 9} (1999), 159--172. 

\bibitem{BCM}
\rm M. Busaniche, L. Cabrer, D. Mundici, 
\it Confluence and combinatorics in finitely generated unital lattice-ordered abelian groups, 
\rm Forum Math. {\bf 24} (2012), 253--271.


\bibitem{DNG}
\rm A. Di Nola, B. Gerla, 
\it Algebras of Lukasiewicz's logic and their semiring reducts,
\rm Contemp. Math. {\bf 377} (2005), 131--144.

\bibitem{droste}
{M.~Droste, W.~Kuich, H.~Vogler (eds.), \emph{Handbook of Weighted Automata}, Springer, \textbf{2009}.}




\bibitem{GRS}
\rm B. Gerla, C. Russo, L. Spada, 
\it Representation of perfect and local MV-algebras,
\rm Math. Slovaca {\bf 61} (2011), 327--340.

\bibitem{GH}
\rm A. M. W. Glass, W. C. Holland, 
\it Lattice-Ordered Groups, 
\rm Kluwer Academic Publishers, {\bf 1989}.

\bibitem{G}
\rm J. S. Golan, 
\it Semirings and Their Applications, 
\rm Kluwer Academic, Dordrecht, {\bf 1999}.

\bibitem{IKN}
S.~N.~Il'in, Y.~Katsov, T.~G.~Nam, \emph{Toward homological structure theory of semimodules: on semirings all of whose cyclic semimodules are projective},
J.~Algebra \textbf{476} (2017), 238--266.


\bibitem{ims}
I.~Itenberg, G.~Mikhalkin, E.~Shustin, \emph{Tropical algebraic geometry} (2nd ed.), Birkh\"{a}user, Basel, \textbf{2009}.

\bibitem{ir}
Z.~Izhakian, L.~Rowen, \emph{Congruences and coordinate semirings of tropical varieties}, Bull.~Sci. Math. \textbf{140(3)} (2016), 231--259.

\bibitem{JK}
\rm J. Je\v zek, T. Kepka, 
\it Finitely generated commutative division semirings, 
\rm Acta Univ. Carolin. Math. Phys. \textbf{51} (2010), 3--27.

\bibitem{JKK}
\rm J. Je\v zek, V. Kala, T. Kepka, 
\it Finitely generated algebraic structures with various divisibility conditions, 
\rm Forum Math. {\bf 24} (2012), 379--397.


\bibitem{Kala15}
V. Kala, \textit{Lattice-ordered abelian groups finitely generated as semirings},
J. Commut. Alg. \textbf{9} (2017), 387--412.

\bibitem{KK}
\rm V. Kala, T. Kepka,
\it A note on finitely generated ideal-simple commutative semirings,
\rm Comment. Math. Univ. Carol. {\bf 49} (2008), 1--9.

\bibitem{KKK} 
\rm V. Kala, T. Kepka, M. Korbel\'a\v r, 
\it Notes on commutative parasemifields, 
\rm Comment. Math. Univ. Carolin. {\bf 50} (2009), 521--533.

\bibitem{KaKo1}
\rm V. Kala, M. Korbel\'a\v r, 
\it Congruence simple subsemirings of $\mathbb Q^+$, 
\rm Semigroup Forum \textbf{81} (2010), 286--296.

\bibitem{KalKor}
V. Kala, M. Korbelář, \textit{Idempotence of finitely generated commutative
semifields}, Forum Math. \textbf{30} (2018), 1461--1474.


\bibitem{KNZ}
Y.~Katsov, T.~G.~Nam, J.~Zumbr\"{a}gel, \emph{On simpleness of semirings and complete semirings}, J.~Algebra Appl. \textbf{13(6)} (2014), 29 pp.

\bibitem{KL}
M. Korbel\'a\v r, G. Landsmann, \textit{One-generated semirings and additive divisibility}, J. Algebra Appl. \textbf{16} (2017), 1750038, 22 pp.


\bibitem{Le} 
E.~Leichtnam, \emph{A classification of the commutative Banach perfect semi-fields of characteristic 1. Applications}, Math. Ann. \textbf{369(1--2)}, 653--703.

\bibitem{litvinov}
\rm G. L. Litvinov,
\it The Maslov dequantization, idempotent and tropical mathematics: a brief introduction,
\rm Idempotent mathematics and mathematical physics, Contemp. Math. \textbf{377} (2005), Amer. Math. Soc., pp.  1–17. \textit{Extended version at arXiv:math/0507014.}


\bibitem{mmr} 
G.~Maze, C.~Monico, J.~Rosenthal, \emph{Public key cryptography based on semigroup actions}, Adv.~Math.~Commun. \textbf{1(4)} (2007), 489--507.

\bibitem{monico}
\rm C. J. Monico,
\it Semirings and semigroup actions in public-key cryptography,
\rm PhD Thesis, University of Notre Dame, USA, {\bf 2002}, vi+61 pp.

\bibitem{M}
\rm D. Mundici, 
\it Interpretation of AF $C^*$-algebras in \L ukasiewicz sentential calculus, 
\rm J. Funct. Anal. {\bf 65} (1986), 15--63.

\bibitem{NL} 
\rm A. Di Nola and A. Lettieri, 
\it Perfect MV-Algebras are Categorically Equivalent to Abelian $\ell$-Groups, 
\rm Studia Logica {\bf 53} (1994), 417--432.


\bibitem{Sima}
L. Šíma, \textit{Finitely generated semirings and semifields}, Master's thesis, Charles University, Czech Republic, \textbf{2021}, iii+31 pp.


\bibitem{W} 
\rm H. J. Weinert, 
\it \" Uber Halbringe und Halbk\" orper. I., 
\rm Acta Math. Acad. Sci. Hungar. {\bf 13} (1962), 365--378. 

\bibitem{WW}
\rm H. J. Weinert, R. Wiegandt,
\it On the structure of semifields and lattice-ordered groups,
\rm Period. Math. Hungar. {\bf 32} (1996), 147--162.

\bibitem{Yang}
\rm Y. Yang,
\it $\ell$-Groups and B\' ezout Domains,
\rm PhD Thesis, Universit\" at Stuttgart, Germany, {\bf 2006}, viii+116 pp.

\bibitem{zumbragel}
\rm J. Zumbr\" agel,
\it Public-key cryptography based on simple semirings,
\rm PhD Thesis, Universit\" at Z\" urich, Switzerland, {\bf 2008}, x+99 pp.
}

\end{thebibliography}
\end{document}